\documentclass[14pt]{amsart}
\address{\newline{\normalsize Kavli IPMU (WPI), The University of Tokyo, 5-1-5 Kashiwanoha, Kashiwa, 277-8583, Japan}
\newline{\it E-mail address}: ilya.karzhemanov@ipmu.jp}
\usepackage{amscd,amsthm,amsmath,amssymb}
\usepackage[dvips]{graphicx}
\usepackage[matrix,arrow]{xy}

\makeatletter\@addtoreset{equation}{section}\makeatother

\makeatletter\@addtoreset{subsection}{equation}\makeatother

\newcommand{\p}{\mathbb{P}}
\newcommand{\cel}{\mathbb{Z}}

\newcommand{\com}{\mathbb{C}}

\newcommand{\slg}{\mathrm{SL}}

\newcommand{\map}{\longrightarrow}
\newcommand{\au}{\mathrm{Aut}}

\newtheorem{theorem}[equation]{Theorem}
\newtheorem{prop}[equation]{Proposition}
\newtheorem{lemma}[equation]{Lemma}
\newtheorem{cor}[equation]{Corollary}

\theoremstyle{remark}
\newtheorem{remark}[equation]{Remark}

\newtheorem*{assumption}{Assumption}
\newtheorem{ex}[equation]{Example}

\begin{document}

\title{On RC varieties without smooth rational curves}

\thanks{}

\author{Ilya Karzhemanov}

\begin{abstract}
We construct normal rationally connected varieties (of arbitrarily
large dimension) not containing any smooth rational curves.
\end{abstract}

\medskip

\thanks{{\it MS 2010 classification}:  14M22, 14H45, 14H60}


\sloppy

\maketitle

\bigskip

\section{Introduction}
\label{section:f}

\refstepcounter{equation}
\subsection{}
\label{subsection:f-1}

Let $V$ be a rationally connected projective variety, not
necessarily smooth or normal, defined over $\com$. Then any smooth
$V$ is covered by various images of morphisms $f: \p^1\map V$ such
that the pullback $f^*T_V$ is ample (see e.g. \cite[Theorem
3.7]{kol}). One also observes that the image $f(\p^1)$ is
\emph{smooth} in this case for a generic choice of $f$ (see
\cite[Theorem 3.14]{kol}).

The present paper grew out of an attempt to understand whether the
preceding property holds for an arbitrary $V$ as well (see
\cite{kebekus-1}, \cite{kebekus-2} for some related results).
Namely, if $V$ is \emph{normal} and rationally connected, is there
always at least one smooth rational curve on $V$? Or more
generally, i.e. dropping the normality assumption, are there such
rationally connected $V$ that don't contain any smooth rational
curves, with $\dim V$ being arbitrarily large?

In the non-normal case, the answer is evident when $\dim V = 1$
(normalization), although even for the product $V \times V$ we are
not sure how to proceed (there might exist $\p^1 \subset V \times
V$ with degree $>1$ projection onto $V$). So the assumption $\dim
V \to \infty$ makes the problem more interesting (we refer to
Section~\ref{section:dis} for further variations).

Our main result treats the normal case as follows:

\begin{theorem}
\label{theorem:main-theorem} There exists normal rationally
connected variety $V$ without smooth rational curves (we call such
$V$ \emph{weird}) and with $\dim V$ arbitrarily large.
\end{theorem}

Let us briefly outline the strategy of the proof of
Theorem~\ref{theorem:main-theorem}.

\refstepcounter{equation}
\subsection{}
\label{subsection:f-2}

To make the construction easier we'd like to consider those $V$
that parameterize certain geometric objects. This should, in
principle, allow one interpret smoothness of any rational curve $Z
\subset V$ in terms of properties of the corresponding family of
objects. We've given our preference to the moduli spaces
$\mathcal{SU}_X(r)$ of rank $r
> 1$ and $\det = \mathcal{O}_X$ vector bundles over an algebraic curve $X$ of genus $g > 1$ (compare with \cite{castravet} for
another usage of $\mathcal{SU}_X(r)$ in birational geometry).

Recall that $\mathcal{SU}_X(r)$ is a \emph{Fano variety} (see
Section~\ref{section:p} for its specific properties). Yet,
unfortunately, it is too early to just set $V :=
\mathcal{SU}_X(r)$ and conclude the proof of
Theorem~\ref{theorem:main-theorem}:

\begin{ex}
\label{example:all-moduli-dont-work} Given any stable vector
bundle $E \in \mathcal{SU}_X(r)$, after twisting $E$ by
$\mathcal{O}_X(m)$ for some $m \in \cel$ (independent of $E$) one
may identify $E \otimes \mathcal{O}_X(m)$ with an extension class
from
$\text{Ext}^1(\mathcal{O}_X(rm),\mathcal{O}_X\otimes\com^{r-1}) =:
L$. This provides a rational dominant map $\p(L) \dashrightarrow
\mathcal{SU}_X(r)$ (see \cite[Proposition 7.9]{drez-nar} for
instance) and shows that $\mathcal{SU}_X(r)$ is in fact
\emph{unirational}. Now comes the discouraging observation that
$\mathcal{SU}_X(r)$ actually contains plenty of smooth rational
curves. Namely, starting with any point $p \in X$, some generic
vector bundle $\mathcal{E}$ with determinant $\det \mathcal{E} =
\mathcal{O}_X(p)$, and a linear form $l$ on the fiber
$\mathcal{E}_p$, we consider the morphism of sheaves $\mathcal{E}
\map \mathcal{O}_p$ corresponding to $l$. The kernel of this
morphism is again a rank $r$ vector bundle; it is stable due to
\cite[Lemma 5.2]{lange} and has trivial determinant. This shows
that the whole $\p^{r-1}\ni l$ embeds into $\mathcal{SU}_X(r)$
(cf. \cite[5.9]{nar-ram}).
\end{ex}

Recall that all the (equivalent) constructions of
$\mathcal{SU}_X(r)$ involve certain GIT quotient of a source space
$\mathbb{S}$ by some group $G$. This $\mathbb{S}$ can be the space
of relative Grassmannians over $X$ (resp. $G = \text{PGL}$), the
space of $\text{SU}_r(\com)$-representations of $\pi_1(X)$ (resp.
$G = \text{SU}$), the space of flat connections on a fixed rank
$r$ topologically trivial vector bundle $E$ over $X$ (resp. $G = $
the gauge group), etc. Our idea then was, starting with a smooth
rational curve $Z \subset \mathcal{SU}_X(r)$, lift it in a $1:1$
manner to $\mathbb{S}$ and obtain a contradiction with the fact
that $\mathbb{S}$ is affine or something like this.

However, as Example~\ref{example:all-moduli-dont-work} shows, this
idea won't work directly. The reason behind is that the claimed
``lifting to $\mathbb{S}$" doesn't exist: there's always an
ambiguity in the choice of a point $\in\mathbb{S}$ to associate
with any given point on $Z$. In order to circumvent this we use
another construction of $\mathcal{SU}_X(r)$ (compare with
\cite[5.1]{tyu}). Namely, after some care (see
Section~\ref{section:p}), one may take $\mathbb{S} = $ the group
of special $(r\times r)$-matrices with coefficients in a formal
power series ring (resp. $G = $ some ind-group). Then for a
particular locus $\mathcal{D}\subset\mathcal{SU}_X(r)$ (see
{\ref{subsection:tp-2}}), with any point on $\mathcal{D}$ one
associates (canonically) a point in $\mathbb{S}$, modulo some
moderate assumptions (cf.
Proposition~\ref{theorem:b-s-basis-for-lambda}).

\begin{remark}
\label{remark:rem-1} This $\mathcal{D}$ seems to be another
natural object, in addition to the theta-divisor $\mathcal{L}$
(see {\ref{subsection:p-3}} below), which comes for free with
$\mathcal{SU}_X(r)$. It would be interesting to explore the
relation between $\mathcal{D}$ and $\mathcal{L}$ further: whether,
say, rationality (resp. precise dimension) statement for
$\mathcal{L}$ (see e.g. \cite[3.1]{tyu}) holds also for
$\mathcal{D}$?
\end{remark}

One observes next that $\mathcal{D}$ is normal, projective and
rationally connected -- the facts we derive from the
ind-construction of $\mathcal{SU}_X(r)$. This readily shows that
there is no smooth rational $Z \subset \mathcal{D}$, since
otherwise it'll be lifted to $\mathbb{S} = $ direct limit of
affine varieties, which is impossible (see
Section~\ref{section:tp} for details).

\bigskip

\section{Preliminaries}
\label{section:p}

\refstepcounter{equation}
\subsection{}
\label{subsection:p-1}

Let the notation be as in {\ref{subsection:f-2}}. We first briefly
recall the ind-construction of the moduli spaces
$\mathcal{SU}_X(r)$ (see \cite{bea-las}, \cite{bea-las-sor},
\cite{faltings} for a complete account). For this we fix a point
$p \in X$, a small disk $\Delta \subset X$ around $p$ and a local
coordinate $z$ on $\Delta$ such that $z(p) = 0$. We also put $K :=
\com((z))$ and $\mathcal{O} := \com[[z]]$.

Let $E$ be a vector bundle on $X$ of rank $r$ and $\det E =
\mathcal{O}_X$. One proves by induction on $r$ that $E$ is trivial
over $X^* := X\setminus p$ (cf. \cite[Lemma 3.5]{bea-las}). Then
trivializing $E$ also over $\Delta$ provides an element
$\gamma\in\slg_r(K)$ whose class in the space $\mathcal{Q} :=
\slg_r(K)\slash\slg_r(\mathcal{O})$ uniquely determines $E$ up to
the left action on $\slg_r(K)$ of the subgroup
$\slg_r(\mathcal{O}_{X^*})$ (i.e. up to a choice of trivialization
for $E\big\vert_{X^*}$).

Equivalently, one may associate with $\gamma$ a \emph{special
lattice} $\Lambda\simeq\mathcal{O}^{\oplus r}$ (with
$\text{Vol}\,\Lambda := \det \gamma = 1$) generated as a
$\mathcal{O}$-module by global sections of $E\big\vert_{X^*}$ and
such that $E$ (or $\gamma$) is reconstructed from $\Lambda$ (see
\cite[Proposition 2.3]{bea-las}). In particular, one regards
$\mathcal{Q}$ as a collection of all such $\Lambda$ (with the
obvious $\slg_r(\mathcal{O}_{X^*})$-action) and represents it as a
direct limit of irreducible projective varieties
$\mathcal{Q}^{(N)},N\ge 0$, which consist of those $\Lambda$ that
have a basis $z^{d_i}e_i$ with $d_i\ge -N,\sum d_i = 0$, and
$e_i\in\mathcal{O}^r$ for all $1\le i\le r$ (see \cite[Theorem
2.5, Proposition 2.6]{bea-las}).

We will need the following simple observation:

\begin{lemma}[{cf. \cite[Lemma 4.5]{bea-las}}]
\label{theorem:can-take-gamma-id} With notation as above, every
$\slg_r(\mathcal{O}_{X^*})$-orbit of $\Lambda$ is represented by
some matrix in $\slg_r(\com[[z^{-1}]])$ (by definition $\Lambda$
is the image of $\gamma$ under the quotient morphism $\slg_r(K)
\map \mathcal{Q}$), so that $\Lambda$ admits a basis consisting of
vectors from $\com[[z^{-1}]]^r$. In particular, one may choose
$\gamma\in\slg_r(\com[[z^{-1}]])$ to satisfy $\gamma = I \mod
z^{-1}$, where $I$ is the identity matrix.
\end{lemma}

\refstepcounter{equation}
\subsection{}
\label{subsection:p-2}

Further, recall that there's a \emph{canonical} isomorphism in
codimension $1$ $\mathcal{SU}_X(r) \simeq
\slg_r(\mathcal{O}_{X^*})\setminus\mathcal{Q}$, where
$\mathcal{SU}_X(r)$ (resp.
$\slg_r(\mathcal{O}_{X^*})\setminus\mathcal{Q}$) is considered as
an algebraic (resp. quotient) stack (see \cite[Proposition 3.4,
Lemma 8.2]{bea-las} and Remark~\ref{remark:sl-x-orbit} below). In
addition, the group $\slg_r(\mathcal{O}_{X^*})$ also carries the
structure of an integral ind-scheme, so that its action on
$\mathcal{Q} = \varinjlim \mathcal{Q}^{(N)}$ is compatible with
the underlying ind-structure (cf. \cite[Proposition
6.4]{bea-las}).

Altogether this yields an ind-structure on $\mathcal{SU}_X(r)$, so
that $\mathcal{SU}_X(r)\subset
\slg_r(\mathcal{O}_{X^*})\setminus\mathcal{Q}$ as an open
ind-subscheme with codimension $>1$ complement, and an ample line
bundle $\mathcal{L}\big\vert_{\mathcal{SU}_X(r)}$ with the
pullback $\pi^*\mathcal{L}$ under the natural projection $\pi:
\mathcal{Q}\map \slg_r(\mathcal{O}_{X^*})\setminus\mathcal{Q}$
being some $\slg_r(\mathcal{O}_{X^*})$-invariant (determinant)
line bundle (see \cite[Sections 5, 7, 8]{bea-las}). The global
sections from $H^0(\mathcal{SU}_X(r),\mathcal{L}^c)$ (\emph{the
space of conformal blocks}), $c\in\cel$, coincide by construction
with the $\slg_r(\mathcal{O}_{X^*})$-invariant global sections of
$\pi^*\mathcal{L}^c$ (see \cite[Theorem 8.5]{bea-las}). We'll
mention a few more properties of $\mathcal{L}$ in
{\ref{subsection:p-3}}.

\begin{remark}
\label{remark:sl-x-orbit} Using the (mappings by) global sections
of $\mathcal{L}^c,c\gg 1$, one may regard both $\mathcal{SU}_X(r)$
and $\slg_r(\mathcal{O}_{X^*})\setminus\mathcal{Q}$ as essentially
the same \emph{scheme} with two complementary structures: of a
quasi-projective variety and an ind-scheme. In particular, any
open (resp. closed) ind-subscheme of
$\slg_r(\mathcal{O}_{X^*})\setminus\mathcal{Q}$ corresponds,
tautologically, to an open (resp. closed) subscheme in
$\mathcal{SU}_X(r)$, and vice versa (compare with the proof of
Theorem 7.7 in \cite{bea-las}).
\end{remark}

\refstepcounter{equation}
\subsection{}
\label{subsection:p-3}

To conclude this section, let's fix a bit more of
notation/conventions and recall some auxiliary facts, as those
will be used later in Section~\ref{section:tp}.

Choose some semi-stable bundle $E\in\mathcal{SU}_X(r)$ (identified
with a point in the corresponding moduli space) with a Hermitian
structure specifying the gauge group action on the space of all
connections on $E$. Using \cite[3.2]{tyu} (cf. \cite[Lemma
4.2]{drez-nar}) and the main result of \cite{don} (cf.
\cite{nar-sesh}) one represents $E$ as a sum of stable bundles.

Note that any two direct sum decompositions $\oplus E_i = E =
\oplus E'_i$ differ only by the order of summands. Indeed,
otherwise arguing by induction on $r$ one obtains a surjection
$E_i \map E'_i$ between two stable bundles $E_i \not\simeq E'_i$
of degree $0$, which is impossible. This leads to the following:

\begin{theorem}
\label{theorem:un-nabla-ss-e} There exists a \emph{unique} (up to
$\text{Aut}\,E$) flat unitary connection $\nabla$ on $E$.
\end{theorem}

Finally, although we'll not quite need this, let us mention for
consistency some other facts about $\mathcal{SU}_X(r)$ (see also
Section~\ref{section:dis}). First of all, variety
$\mathcal{SU}_X(r)$ is \emph{locally factorial}, with Picard group
generated by $\mathcal{L}$ (see \cite[Theorems A,\,B]{drez-nar}).
One may equivalently interpret $\mathcal{L}$ as the so-called
\emph{theta-divisor}. The latter consists of all
$E\in\mathcal{SU}_X(r)$ for which $H^0(X,E\otimes\xi)\ne 0$ with
respect to some fixed cycle $\xi$ on $X$ of degree $g - 1$.
Furthermore, the canonical class $K_{\mathcal{SU}_X(r)}$ equals
$-2r\mathcal{L}$, so that $\mathcal{SU}_X(r)$ is a Fano variety.

\bigskip

\section{Proof of Theorem~\ref{theorem:main-theorem}}
\label{section:tp}

\refstepcounter{equation}
\subsection{}
\label{subsection:tp-1}

We retain the notation of Section~\ref{section:p}. Put $t :=
z^{-1}$ and let $\gamma$ (resp. $\Lambda$) be as in
Lemma~\ref{theorem:can-take-gamma-id}. Recall that
$E\big\vert_{\Delta} = \com^r\times\Delta$ with constant basis,
while the $\mathcal{O}_{X^*}$-module $E\big\vert_{X^*}$ is
generated by $\Lambda$, so that the two data are glued over
$\Delta\cap X^*$ via $\gamma$. We will additionally assume that
$\gamma = I \mod t^2$.

To apply the strategy outlined in {\ref{subsection:f-2}} one
should have (at least) the following:

\begin{prop}
\label{theorem:b-s-basis-for-lambda} In the previous setting,
$\Lambda$ carries a \emph{unique} (Bohr-Sommerfeld) basis, which
does not depend on the choice of $\gamma$. In fact, there is only
one $\gamma$ associated with $\Lambda$ (cf.
{\ref{subsection:p-1}}), satisfying $\gamma = I \mod t^2$.
\end{prop}

\begin{proof}
According to our assumption there exists a collection of vectors
$\{e_1,\ldots,e_r\}\subset\com[[t]]^r$ that generates
$E\big\vert_{X^*}$ and coincides with the standard basis of
$\com^r$ modulo $t^2$. Using this, the flat unitary connection
$\nabla$ on $E$ (see Theorem~\ref{theorem:un-nabla-ss-e}), plus
the preceding description of $E$ in terms of $\Delta,X^*$ and
$\gamma$, we construct (via the parallel transport starting with
$e_i(0)$) global $C^{\infty}$-sections $\varepsilon_i$ of $E,1\le
i\le r$, which generate $E\big\vert_{X^*}$ and satisfy the
equation $\nabla\varepsilon_i = 0$ in a small disk $\Delta_0
\subset X^*$ around $t = 0$.

More precisely, since $\nabla$ is flat and $de_i(0) =
0$,\footnote{Here $d$ is the usual K\"ahlher differential of the
ring $\com[[t]]$.} equation $\nabla\varepsilon_i = 0$ and its
solutions $\{\varepsilon_1,\ldots,\varepsilon_r\}$ do not depend
on the choice of $e_i$ (aka $\gamma$). In particular, there is a
correctly defined (canonical) extension of every $\varepsilon_i$
over the entire $X$, as claimed. We also observe that by
uniqueness of the solutions all entries in $\varepsilon_i$ are
some elements from $\com[[t]]$ because $\nabla^{0,1} =
\bar{\partial}$ locally on $\Delta_0$. This gives the desired
basis for $\Lambda$.

Recall next that $\nabla$ is unique up to $\text{Aut}\,E$. Now the
last claim of proposition follows from the fact that one may take
$\{\varepsilon_1,\ldots,\varepsilon_r\}$ to be the columns of
$\gamma$ and that $g \gamma g^{-1} = \gamma$ for any $g \in
\text{Aut}\,E$ by definition.
\end{proof}

\begin{remark}
\label{remark:con-from-th} (Notation as in the proof of
Proposition~\ref{theorem:b-s-basis-for-lambda}) The sections
$\varepsilon_i$ provide an isomorphism $E \simeq
\mathcal{O}_X^{\oplus r}$ of $C^{\infty}$-bundles.\footnote{The
bundle $\mathcal{O}_X^{\oplus r}$ carries a natural Hermitian
structure which varies together with $E$.} Let $\mathcal{D}
\subset \mathcal{SU}_X(r)$ be the moduli space of all
\emph{completely reducible} flat connections on
$\mathcal{O}_X^{\oplus r}$. Then conversely, for any
$C^{\infty}$-trivial $E \in \mathcal{D}$ the corresponding
$\Lambda,\gamma$ can be chosen to satisfy $\gamma = I \mod t^2$.
Indeed, any such $E$ is generated by $r$ $\nabla$-flat
$C^{\infty}$-sections $\varepsilon_i$, obtained from some constant
sections $\varepsilon^0_i$ of $\mathcal{O}_X^{\oplus r}$ via
fiberwise (gauge) transformation. Locally on $\Delta_0$ one has
$\varepsilon_i = \varepsilon^0_i$ modulo $(t,\bar{t})^2$ by
construction, which implies that $d\varepsilon_i(0) = 0 =
\nabla\varepsilon_i$. Then all $\varepsilon_i$ depend only on $t$
due to $\nabla^{0,1} = \bar{\partial}$ and uniqueness of the
solutions. Note that this construction of $\varepsilon_i$ requires
$\Delta_0\ni 0$ to be \emph{fixed}. It also shows that any other
lattice $g\cdot\Lambda$, for $g\in\slg_r(\mathcal{O}_{X^*})$, has
the corresponding $g\gamma g^{-1}$ equal to $I \mod t^2$ as well.
Indeed, recall that the condition for $\varepsilon_i$ to be
$\nabla$-flat translates into $\varepsilon_i = \varepsilon^0_i$
modulo $(t,\bar{t})^2$, which must be gauge invariant.
\end{remark}

\refstepcounter{equation}
\subsection{}
\label{subsection:tp-2}

The locus $\mathcal{D}\subset \mathcal{SU}_X(r)$ from
Remark~\ref{remark:con-from-th} is our candidate for the weird
rationally connected variety $V$. Let's study the geometry of
$\mathcal{D}$ more closely by employing its description in terms
of the lattices $\Lambda$. But first we make the following:

\begin{assumption}
Fix $r\ge 3$ and choose the initial curve $X$ to be generic of
genus $\ge 3$. This gives $\au\,E = \com^*$ for generic
$E\in\mathcal{D}$ (take for instance $E = $ the direct sum of
different line bundles $\delta_i \ne \mathcal{O}_X$ satisfying
$\deg \delta_i = 0$ and $\delta_i \ne -\delta_j$ for all $i,j$).
\end{assumption}

We will write $\mathcal{SU}_X(r) =
\slg_r(\mathcal{O}_{X^*})\setminus\mathcal{Q}$ in what follows
(cf. {\ref{subsection:p-2}} and Remark~\ref{remark:sl-x-orbit}).
This aims to just simplify the notation and won't cause any loss
of generality.

Consider the subset $\mathcal{D}_0\subset\mathcal{Q}$ of all
lattices $\Lambda$ satisfying $\gamma = I \mod t^2$ as in
{\ref{subsection:tp-1}}. Let also
$\mathcal{M}\subset\slg_r(\com[[t]])$ be the locus of all matrices
$I + \displaystyle\sum_{i=2}^{\infty}t^i\Theta_i$ for various
$\Theta_i\in\text{M}_r(\com)$. It follows from the previous
constructions that $\mathcal{M}$ maps onto $\mathcal{D}_0$ under
the quotient morphism $\slg_r(K)\map\mathcal{Q} =
\slg_r(K)\slash\slg_r(\mathcal{O})$.

\begin{lemma}
\label{theorem:d-o-q} $\mathcal{D}_0\subset\mathcal{Q}$ is a
$\slg_r(\mathcal{O}_{X^*})$-invariant closed ind-subscheme.
\end{lemma}

\begin{proof}
Consider some $\Lambda\in\mathcal{D}_0$. Recall that $\Lambda$
uniquely determines the corresponding $\gamma\in\mathcal{M}$ (see
Proposition~\ref{theorem:b-s-basis-for-lambda}). Then, since the
group $\slg_r(\mathcal{O})$ acts freely on
$\slg_r(K)\supset\mathcal{M}$, the set $\mathcal{D}_0$ may be
identified with $\mathcal{M}$ near $\Lambda$.

In particular, any (bounded in analytic topology) Cauchy sequence
of lattices $\Lambda_i\in\mathcal{D}_0$ yields the corresponding
sequence of $\gamma_i\in\mathcal{M}$, as follows from the
definition of $\slg_r(K) \map \mathcal{Q}$. One then (obviously)
has a limit $\lim \gamma_i\in\mathcal{M}$ that maps to
$\lim\Lambda_i\in\mathcal{D}_0$. This shows that $\mathcal{D}_0$
is closed in analytic topology. Now using the representation
$\mathcal{Q} = \varinjlim \mathcal{Q}^{(N)}$ from
{\ref{subsection:p-2}} we obtain that
$\mathcal{D}_0\subset\mathcal{Q}$ is a closed ind-subscheme.

Finally, the $\slg_r(\mathcal{O}_{X^*})$-invariance of
$\mathcal{D}_0$ follows from Remark~\ref{remark:con-from-th},
which concludes the proof.
\end{proof}

From Lemma~\ref{theorem:d-o-q} we obtain
\begin{equation}
\nonumber \mathcal{D} = \pi(\mathcal{D}_0) =
\slg_r(\mathcal{O}_{X^*})\setminus\mathcal{D}_0
\end{equation}
for the natural $\slg_r(\mathcal{O}_{X^*})$-equivariant morphism
$\pi: \mathcal{Q} \map \mathcal{SU}_X(r)$ of ind-schemes.

\begin{prop}
\label{theorem:props-of-d} $\mathcal{D}$ is a normal, projective
and rationally connected variety.
\end{prop}

\begin{proof}
It follows from the above {\it Assumption} that the group
$\slg_r(\mathcal{O}_{X^*})$ acts freely at the general point on
$\mathcal{D}_0$. This gives

\begin{lemma}
\label{theorem:lemma-d-c-i} $\mathcal{D}$ is irreducible and
reduced (i.e. integral).
\end{lemma}

\begin{proof}
Recall that $\mathcal{D}_0$ and $\mathcal{M}$ are isomorphic at
the general point (cf. the proof of Lemma~\ref{theorem:d-o-q}).
Note also that $\mathcal{M}$ is integral (compare with the proof
of \cite[Proposition 2.6]{bea-las}). Hence $\mathcal{D}_0$ is
integral as well. Then, since $\mathcal{D}_0 \simeq
\mathcal{D}\times\slg_r(\mathcal{O}_{X^*})$ (equivariantly) at the
general point, the claim follows from \cite[Proposition 6.4, Lemma
6.3,\,d)]{bea-las}.
\end{proof}

From the definition of $\pi: \mathcal{D}_0 \map \mathcal{D}$ we
deduce that the locus $\mathcal{D}\subset\mathcal{SU}_X(r)$ is
closed in analytic topology (cf. Lemma~\ref{theorem:d-o-q} and
Remark~\ref{remark:sl-x-orbit}). Hence $\mathcal{D}$ is a
projective integral variety.

Further, consider some lattice $\Lambda\in\mathcal{D}_0$,
identified with the matrix $\gamma\in\mathcal{M}$ as in the proof
of Lemma~\ref{theorem:d-o-q}. Choose another lattice
$\Lambda'\in\mathcal{D}_0$, with associated matrix $\gamma'$, and
define the curve
$$
C := \{(1-x)\gamma + x\gamma'\ \vert \
x\in\com\}\subset\mathcal{D}_0.
$$

\begin{lemma}
\label{theorem:c-is-rat} $\pi(C)$ is a rational curve.
\end{lemma}

\begin{proof}
Indeed, there's an obvious birational map $\p^1 \dashrightarrow C$
of ind-schemes (for $\p^1 = \varinjlim \p^1$ tautologically),
which yields a rational dominant map $\p^1 \dashrightarrow
\pi(C)$.
\end{proof}

It follows from Lemmas~\ref{theorem:lemma-d-c-i} and
\ref{theorem:c-is-rat} that $\mathcal{D}$ is rationally connected
(cf. \cite[Ch.\,IV, Proposition 3.6]{kol}). Thus it remains to
prove normality.

Firstly, the proof of \cite[Proposition 6.1]{bea-las} shows that
$\mathcal{M}$ is smooth in codimension $1$, which implies that
$\mathcal{M}$ is normal because it is a complete intersection on
$\slg_r(\com[[t]])\subset\slg_r(K)$ (more precisely, $\mathcal{M}$
is a direct limit of finite-dimensional complete intersections,
which are smooth in codimension $1$, hence normal by Serre's
criterion). Then $\mathcal{D}_0\subset\mathcal{Q}$ is also normal
for $\slg_r(\mathcal{O})$ acting freely on
$\slg_r(K)\supset\mathcal{M}$.

Secondly, since $\text{rk}\,\Lambda > 1$ for any lattice
$\Lambda\in\mathcal{D}_0$, the condition $\Lambda = g\cdot\Lambda$
for some $g\in\slg_r(\mathcal{O}_{X^*})\setminus{\{\text{id}\}}$
is of codimension $>1$ on $\mathcal{D}_0$. Thus we get
$\mathcal{D}_0 = \mathcal{D}\times\slg_r(\mathcal{O}_{X^*})$ and
$\mathcal{D}$ is normal -- all in codimension $1$.

Finally, if $\mathcal{L}$ is the ample generator of
$\text{Pic}\,\mathcal{SU}_X(r)$ (see {\ref{subsection:p-3}}), it
follows from the quotient construction of $\mathcal{D}$ that any
function in $\com(\mathcal{D})$ can be represented as a ratio of
global sections of various restrictions
$\mathcal{L}^c\big\vert_{\mathcal{D}},c\gg 1$. In particular, the
$S_2$ property is satisfied for all the rational functions on
$\mathcal{D}$, and so $\mathcal{D}$ is normal by Serre's
criterion.

Proposition~\ref{theorem:props-of-d} is completely proved.
\end{proof}

From (the proof of) Proposition~\ref{theorem:props-of-d} we deduce

\begin{cor}
\label{theorem:dim-of-d} $\mathcal{D}\subset\mathcal{SU}_X(r)$ has
codimension $\le r^2$ (hence in particular $\dim \mathcal{D}\ge
(r^2 - 1)(g - 1) - r^2$ can be made arbitrarily large).
\end{cor}

\begin{proof}
This follows from the fact that the locus
$\mathcal{M}\subset\slg_r(\com[[t]])$ is defined by the equation
$\Theta_1 = 0$ for generic matrix $I + \displaystyle\sum_{i =
1}^{\infty}t^i\Theta_i\in\slg_r(\com[[t]])$ (cf.
{\ref{subsection:tp-2}} and
Lemma~\ref{theorem:can-take-gamma-id}).
\end{proof}

\refstepcounter{equation}
\subsection{}
\label{subsection:tp-3}

Now let's turn to the proof of Theorem~\ref{theorem:main-theorem}.
Suppose that $Z \subset \mathcal{D}$ is a smooth rational curve.
Associating with any point $E \in Z$ the lattice $\Lambda$ as in
{\ref{subsection:p-1}} yields an affine bundle $\mathcal{V}$ over
$Z$. More specifically, since every $\Lambda$ carries the
canonical basis of Proposition~\ref{theorem:b-s-basis-for-lambda},
using it one identifies $\Lambda$ with the affine space $\com^r$.

Further, since $Z = \p^1$ and the bundle $\mathcal{V}$ is locally
analytically trivial by construction, we get $\mathcal{V} =
\displaystyle\bigoplus_{i=1}^r\mathcal{O}_{\p^1}(d_i)$ for some
$d_i\in\cel$. Note that $\sum d_i = \deg \mathcal{V} = 0$ for
$\text{Vol}\,\Lambda = 1$. Then taking a global section of
$\mathcal{V}$ gives an embedding $Z \subset \mathcal{D}_0$.

Thus we obtain a non-constant family of matrices $\gamma \in
\slg_r(K)$ \emph{algebraically} parameterized by $Z$. In
particular, there exists a rational function $f\in\com(Z)$ such
that $f(z)\ne\infty$ for all $z \in Z$, which is absurd. The proof
of Theorem~\ref{theorem:main-theorem} is complete (cf.
Proposition~\ref{theorem:props-of-d} and
Corollary~\ref{theorem:dim-of-d}).

\bigskip

\section{Some questions and comments}
\label{section:dis}

We'd like to conclude the paper by asking the following questions:

\begin{itemize}

\item Is there a weird $V$ over a field of positive characteristic
(cf. \cite{bo-yu}, \cite{bo-yu-ha})? Similarly, modifying the
notion of rational connectivity accordingly, is there a
non-K\"ahler compact complex (maybe even smooth) weird $V$ (cf.
\cite{verb})?

\smallskip

\item Is the locus $\mathcal{D}\subset\mathcal{SU}_X(r)$ a Fano
variety (and how to express it in terms of theta-divisors)? Does
it have locally factorial singularities (resp. what is its Picard
group)? Same questions for any weird $V$.

\smallskip

\item By applying the weak factorization theorem it would be interesting
to find out whether being weird provides an obstruction for
variety to be rational. What about the case of $\mathcal{D}$ again
(cf. Remark~\ref{remark:rem-1})?

\smallskip

\item Note that the locus $\mathcal{D}$ consists entirely of \emph{strictly}
semi-stable bundles (cf. Remark~\ref{remark:con-from-th}). In
particular, when $r = 2$, writing any $E \in \mathcal{D}$ as a sum
of two line bundles one finds that there is a $2:1$-cover
$\text{Pic}^0(X) \map \mathcal{D}$ (i.e. $\mathcal{D}$ is the
\emph{Kummer variety}).\footnote{This example shows that the {\it
Assumption} made in {\ref{subsection:tp-2}} is crucial for
$\mathcal{D}$ to be rationally connected. Indeed, for $r = 2$ the
locus $\mathcal{D}$ is not rationally connected (as
$h^0(\mathcal{D},K_{\mathcal{D}}^m) = 1,m\gg 1$), and the reason
is that one lacks the canonical correspondence $\Lambda
\leftrightarrow \gamma$ here (cf. the proof of
Lemmas~\ref{theorem:lemma-d-c-i} and \ref{theorem:c-is-rat}).} Is
the same true for an arbitrary $r > 2$ (with $\text{Pic}^0(X)$
replaced by an appropriate Abelian variety and ``$2:1$" by
``generically Galois")? Similar question for any weird $V$.

\smallskip

\item Is it possible to find weird $V$ in \emph{any} given
dimension (cf. Corollary~\ref{theorem:dim-of-d})?

\end{itemize}

\bigskip

\thanks{{\bf Acknowledgments.}
I'd like to thank A. Beauville, S. Galkin, J. Koll\'ar, and T.
Milanov for their interest, valuable comments and references. Most
of the paper was written during my visits to MIT (US), UOttawa
(Canada) and PUC (Chile) in April-May 2015. I am grateful to these
Institutions and people there for hospitality. The work was
supported by World Premier International Research Initiative
(WPI), MEXT, Japan, and Grant-in-Aid for Scientific Research
(26887009) from Japan Mathematical Society (Kakenhi).

\end{document}